\documentclass[12pt,a4paper]{amsart}
\usepackage[cmtip,arrow]{xy}
\usepackage{pb-diagram,pb-xy}
\usepackage{amsxtra}
\usepackage{enumitem, fancyref, theoremref}
\usepackage{amsmath, amssymb, tikz, amsthm, thmtools}
\usepackage{amsrefs, latexsym,pstricks}
\usepackage{hyperref, caption}

\calclayout

\DeclareMathSymbol{\subsetneq}{\mathrel}{AMSb}{"28}
\DeclareMathSymbol{\rightrightarrows}{\mathrel}{AMSa}{"13}

\newcommand{\BHZ}{\mathbb{Z} + \frac{1}{2}} %half-integers
\newcommand{\BN}{\mathbb{N}} %natural numbers
\newcommand{\BZ}{\mathbb{Z}} %integer numbers
\newcommand{\BC}{\mathbb{C}} %complex numbers
\newcommand{\BY}{\mathbb{Y}} %Young diagram space
\newcommand{\MY}{\Lambda^{\frac{\infty}{2}}V} %half-infinity wedge product
\newcommand{\FM}{\mathfrak{M}} %Schur measure
\newcommand{\qet}[1]{| #1 \rangle} %qet vector notation
\newcommand{\braqet}[2]{\langle #1 | #2 \rangle} %braqet
\newcommand{\triqet}[3]{\langle #1 | #2 | #3 \rangle} %triqet
\newcommand{\vacuum}{\qet{\emptyset}} %vacuum vector
\newcommand{\qetl}{\qet{\lambda}} %qet lambda
\newcommand{\BP}{\mathbb{P}} %probability
\newcommand{\vspan}[1]{\mathrm{span}\{ #1 \}} %span
\newcommand{\column}{\mathrm{column}} %column
\newcommand{\row}{\mathrm{row}} %row
\newcommand{\particles}{\mathrm{Particles}} %particles domain
\newcommand{\holes}{\mathrm{Holes}} %holes
\newcommand{\SL}{\mathfrak{sl}} %sl
\newcommand{\Vir}{\mathrm{Vir}} %Vir
\newcommand{\Kerov}{\mathrm{Kerov}} %kerov representation
\newcommand{\LV}{\widetilde{L}} %modified Virasoro operators
\newcommand{\Ker}{\mathrm{Ker}} %kernel
\newcommand{\LBox}{\blacksquare} %Left Box
\newcommand{\Start}{\mathrm{Start}} %Start

\newcommand{\g}{\mathfrak{g}} %g - Lie algebra
\newcommand{\n}{\mathfrak{n}} %n - Borel subalgebra
\newcommand{\h}{\mathfrak{h}} %h - Cartan subalgebra

\theoremstyle{plain}

\newtheorem{prop}{Proposition}
\newtheorem{theorem}{Theorem}[section]
\newtheorem{lemma}{Lemma}[section]

\theoremstyle{remark}

\newtheorem{rem}{Remark}

\theoremstyle{definition}
\newtheorem{definition}{Definition}

\begin{document}

\title{Probabilistic Aspects of the Theory of Vertex Algebras}
\author{Dmitry Golubenko}

\address{Faculty of Mathematics, Higher School of Economics, 7~Vavilova str., Moscow, Russia, 117312}
\email{golubenko@mccme.ru}

\begin{abstract}
Determinantal processes on half-integer line can be studied using vertex algebras. They were used by Okounkov in \cite{oko2}, where Schur processes were introduced and proved to be determinantal. We want to extend this vertex algebra approach. First, we establish the connection between the so-called \emph{z-measures} and Virasoro operators. In fact, we prove that z-measures can be established by Virasoro algrebra action on Young diagrams space. Second, we introduce Virasoro measures and prove their determinancy.
\end{abstract}

\maketitle

\tableofcontents

\section{Introduction}
This work deals with Schur measures and vertex algebra structures associated with them. The
Schur measures are (complex-valued) probability measures on the set of all Young diagrams defined
as
\begin{equation}
\BP(\lambda) = \frac{1}{Z}s_{\lambda}(x_1, x_2, ...)s_{\lambda}(y_1, y_2, ...)
\end{equation}
where $\lambda$ runs over all Young diagrams, $s_{\lambda}$ are the Schur symmetric
functions, $Z$ is the normalization constant, and $\{x_i\}$ and $\{y_i\}$ are two sets of complex variables.
These measures were introduced by Okounkov in a 1999 preprint [Oko01a], where the determinantal
structure of them was also established, and the determinantal correlation kernel was computed.
Since then, the Schur measures have found nice generalizations (for instance, Schur \cite{oro} and
Macdonald \cite{bc} processes and their variants), and have provided an algebraic structure behind
many integrable random systems such as Plancherel random partitions (related to the distribution of
longest increasing subsequences in random permutations), random plane partitions, etc. Certain stochastic
dynamics on Schur measures and Schur processes is an instance of a 2-dimensional anisotropic
Kardar-Parisi-Zhang random growth. The algebraic nature of the probability distributions
allows to establish fine asymptotic processes of the associated random systems --- most notably, the
convergence to the universal Tracy–Widom distributions which manifest the Kardar-Parisi-Zhang
universality of the systems.

This paper is organized as follows. In chapter 2 we define some basic objects such as Kerov operators and modified Virasoro algebra. In chapter 3 we characterise Kerov representation of $\mathfrak{sl}_2$ and prove that this is irreducible in most cases. Then, in chapter 4, we discover that Kerov operators can be described by Virasoro operators.
In chapters 5 and 6, we introduce M-Virasoro processes, the generalization of Schur measures, and prove that they are determinantal; moreover, they can be expressed as Schur measures of some parameters $\{X_i, Y_i\}$; for Virasoro processes it's also shown that $X_i$ and $Y_i$ are linear functions of $z$.

\section{Kerov operators and Virasoro algebra}
\medskip
	We consider the space of Young diagrams $\BY$, see \cite{fulton} for further definitions. For $V = \vspan{\underline{k}|k \in \BHZ}$ consider the subspace $\Lambda^{\frac{\infty}{2}}V$ spanned by $\underline{\xi_1} \wedge \underline{\xi_2} \wedge ... \wedge \underline{\xi_n} \wedge ... \in \MY$ such that $\xi_1 > \xi_2 > ... > \xi_n > ...$ and this sequence containts $(-\infty, N]$ for some $N \in \BHZ$. These basis vectors can be parametrized by Young diagrams this way:
	\begin{equation}
	\BY \ni \lambda \leftrightarrow (\lambda_1 - \frac12) \wedge (\lambda_2 - \frac32) \wedge ... \wedge (\lambda_{|\lambda|} - |\lambda| + \frac12) \wedge (-|\lambda| - \frac12) \wedge (-|\lambda| - \frac32) \wedge ... \in \MY 
	\end{equation}
	Therefore $\qetl = (\lambda_1 - \frac12) \wedge (\lambda_2 - \frac32) \wedge ... \wedge (\lambda_{|\lambda|} - |\lambda| + \frac12) \wedge (-|\lambda| - \frac12) \wedge (-|\lambda| - \frac32) \wedge ...$ for every $\lambda \in \BY$. Let us denote $\mathrm{Conf}(\lambda) = \{\lambda_i - i + \frac12\}_{i \in [1; |\lambda|]} \bigcup (-\infty, -|\lambda|)$.
	\begin{definition}
		We say that $\lambda$ has a \emph{particle} in $x \in \BHZ$ if  
		$x \in \mathrm{Conf}(\lambda)$. Otherwise, we say $\lambda$ has a \emph{hole} in $x \in \BHZ$. We say that $\qetl$ has a particle/hole in $x \in \BHZ$ if $\lambda$ has particle/hole there.
	\end{definition}
	
	Let us consider $\MY = \vspan{\qet{\lambda}|\lambda \in \BY}$ as a linear space spanned by vectors parametrized by Young diagrams.
	
	\begin{definition}
		For $\Box \in \lambda$ - a box in a Young diagram $c(\Box)$ is {\it box containment}, which is defined as
		\begin{equation}
		\column(\Box) - \row(\Box)
		\end{equation}	
	\end{definition}
	\begin{definition}[\cite{petrov}, \cite{oko1}] 
		{\it Kerov operators} are linear operators $U,~L,~D$ on $\MY$ defined by the following formulas
		\begin{equation}
		\begin{cases}
		U \qetl = \sum_{\mu = \lambda + \Box} (z + c(\Box)) \qet{\mu} \\
		L \qetl = (zw + 2|\lambda|)\qetl\\
		D \qetl = \sum_{\mu = \lambda - \Box} (w + c(\Box)) \qet{\mu}\\
		\end{cases}
		\end{equation}
		for  $z, w \in \BC$. Note that these operators form an $\SL_2$ triple.
	\end{definition}
	
	It's straightforward that Kerov operators form an $\SL_2$-triple. That defines an $\SL_2$ representation in $\vspan{\qet{\lambda}|\lambda \in \BY}$, which we will call the {\it Kerov representation} $\Kerov(z,w)$.
	
	In \cite{oko1} the generalization of Kerov operators is introduced.
	\begin{definition}
		{\it Rim-hook} of a Young diagram $\lambda$ is a skew diagram $\lambda/\mu$ which is connected an lies on the rim of $\lambda$.
	\end{definition}
	\begin{definition}\cite{oko1}\label{RHKerov}
		{\it Rim-hooked Kerov operators} are the operators $U_r,~L_r,~ D_r$ on $\MY$ induced from linear operators $u_r,~l_r,~d_r$ defined on $V$ such that
		\begin{equation}
		\begin{cases}
		U_r v_k = (z + \frac{k}{r} + \frac12) v_{k+r} \\
		L_r v_k = (z + w + 2\frac{k}{r}) v_{k} \\
		D_r v_k = (w + \frac{k}{r} - \frac12) v_{k-r} \\
		\end{cases}
		\end{equation}
		They form rim-hook Kerov representation $\mathrm{RHKerov}(z,w)$. These operators satisfy the same $\SL_2$ commutation relations.
	\end{definition}
	
	We will figure out the exact formula of rim-hook Kerov operators action on Young diagrams but we'll do it later in this article.
	
	According to \cite{deter}, we call measure on $\BY$ {\it determinantal} (or \emph{determinantal process}) with correlation kernel $K(\cdot, \cdot)$ if its correlation functions are given by
	\begin{center}
		$\BP (\{\lambda|\{x_1, ... x_N\} \subset \mathrm{Conf}(\lambda)\}) = \det [K(x_i, x_j)]_{i,j \in [1,N]}$
	\end{center}
	
	As usual, $s_{\lambda}(\{x_i\}_{i \in \BN})$ stands for Schur polynomials, defined on Young diagrams of shape $(N)$ as coefficients of $\exp(\sum_i x_i z^i)$ as a function of $z$ or just as
	\begin{equation}
	s_{(N)} = s_N = \sum_{\sum i k_i = N} \frac{x_1^{k_1}}{k_1!}...\frac{x_N^{k_N}}{k_N!}
	\end{equation}
	and everywhere else by Jacobi-Trudy identity according to \cite{kacraina}.
	\begin{definition}[\cite{oko2}]
		{\it Schur measure} or {\it Schur process} is the measure on Young diagrams defined by
		\begin{equation}
		\FM (\lambda) = \frac{1}{Z}s_{\lambda}(\{x_i\})s_{\lambda}(\{y_i\})
		\end{equation}
		where $Z = \prod_{i,j}(1 - x_i y_j)^{-1}$ is the partition function. Here $\{x_i\}$ and $\{y_i\}$ are two infinite sequences of complex numbers.
	\end{definition}
	
	From \cite{kacraina} we use the notion of Heisenberg algebra and modified Virasoro algebra. On $\MY$ we have creating operators defined by the following
	\begin{equation}
	\psi_x \underline{\xi_1} \wedge \underline{\xi_2} \wedge ... \wedge \underline{\xi_n} \wedge ... = \underline{x} \wedge\underline{\xi_1} \wedge \underline{\xi_2} \wedge ... \wedge \underline{\xi_n} \wedge ...,~ x \in \BHZ
	\end{equation}
	and annihilating operators $\psi^*_x$, which are dual to creating operators w. r. t. the standart scalar product on $\MY$ which is 
	\begin{equation}
	\braqet{\lambda}{\mu} = 
	\begin{cases}
	1, \lambda = \mu\\
	0, \mathrm{otherwise}
	\end{cases}
	\end{equation}
	\begin{definition}
		{\it Heisenberg algebra} is an algebra spanned by the operators $a_i$ satisfying
		\begin{equation}
		[a_n, a_m] = n\delta_{m+n,0}
		\end{equation}
		They can be realised through creating and annihilating operators:
		\begin{equation}
		a_k = \sum_{x \in \BZ} \psi_{x - k} \psi^*_x,~ k\neq 0
		\end{equation}
		and $a_0$ is the central element of Heisenberg algebra and so acts on Young diagrams by scalar.
	\end{definition}
	\begin{definition}
		{\it Modified Virasoro algbera} is an algebra spanned by the operators
		\begin{equation}
		\LV_{k}(\alpha, \beta) = i\beta ka_k + \frac{1}{2}\sum_{j \in \BZ} :a_{j} a_{k-j}:, ~ k \in \BZ
		\end{equation}
		for $k \neq 0$ and
		\begin{equation}
		\LV_0 = (\alpha^2 + \beta^2) + \sum_{j > 0} a_{-j} a_{j}
		\end{equation}
		with $\alpha, \beta \in \BC$. Here $a_0 \qetl = \alpha \qetl$ and
		\begin{equation}
		:a_k a_m: =
		\begin{cases}
		a_k a_m,~ m \leqslant k\\
		a_m a_k,~ m \geqslant k
		\end{cases}
		\end{equation}
		is the normal ordering. Note that
		\begin{equation}
		[\LV_m, \LV_n] = (m - n)\LV_{m+n} + \delta_{m+n,0}\frac{m^3 - m}{12}(1+12\beta^2)
		\end{equation}
	\end{definition}
	\begin{rem}
		If $k \neq 0$, we can omit the normal ordering because of $[a_m, a_n] = 0,~ m + n \neq 0$.
	\end{rem}
	Having Heisenberg algebra we may redefine Schur measure as
	\begin{eqnarray}
	\FM(\lambda) = \frac{1}{Z} \triqet{\lambda}{\exp(\sum_{i \i \BN} x_i a_{-i})}{\vacuum}\triqet{\emptyset}{\exp(\sum_{i \i \BN} y_i a_{i})}{\lambda} 
	\end{eqnarray} 
  
\section{Decomposition of Kerov representation}
The goal of this paragraph is to prove this
\begin{theorem}
	\begin{itemize}
		\item If $z, w \neq 0$ then Kerov representation can be decomposed into sum of Verma modules
		\begin{equation}
		\Kerov(z,w) = M_{zw} \bigoplus_{N \in \BN, N > 1} |\BY_{N-1}| M_{zw + 2N}
		\end{equation}
		\label{decompose}
		\item If $w, z = 0$ then Kerov representation can be decomposed into sum of one one-dimensional module and Verma modules
		\begin{equation}
		\Kerov(z,w) = \BC \vacuum \bigoplus_{N \in \BN, N > 1} |\BY_{N-1}| M_{2N}
		\end{equation}
		\item If $w \neq 0, z = 0$ then Kerov representation can be decomposed
		\begin{equation}
		\Kerov(z,w) = \frac{U(\SL_2)_{\Kerov}\vacuum \oplus U(\SL_2)_{\Kerov}\qet{\Box}}{U \vacuum = 0,~ D\qet{\Box} = \vacuum, D \vacuum = 0} \bigoplus_{N \in \BN, N > 1} |\BY_{N-1}| M_{2N}
		\end{equation}
		\item If $z \neq 0, w = 0$ then Kerov representation can be decomposed
		\begin{equation}
		\Kerov(z,w) = \frac{U(\SL_2)_{\Kerov}\vacuum \oplus U(\SL_2)_{\Kerov}\qet{\Box}}{U \vacuum = \qet{\Box},~ D\qet{\Box} = 0, D \vacuum = 0} \bigoplus_{N \in \BN, N > 1} |\BY_{N-1}| M_{2N}
		\end{equation}	
	\end{itemize}
\end{theorem}
Firstly, we'll prove these two lemmas.
\begin{lemma}
	$\mathrm{rk} D|_{\BY_N} = |\BY_{N-1}|$ for all $w \in \BC$.	
\end{lemma}
\begin{lemma}
	$U$ has trivial kernel and Verma modules ${U(\SL_2)}_{\Kerov} v_N,~ v_N \in \Ker D|_{\BY_N}$ can be generated from $\Ker D$ basis. Here ${U(\SL_2)}_{\Kerov}$ is the universal enveloping algebra generated by Kerov operators.
\end{lemma}
\subsection{Kernel of $D$}
It's obvious that kernel has a natural grading: $\Ker D = \BC\vacuum \bigoplus_{N \in \BZ} \Ker_N$, where $\Ker_N := \Ker f|_{\BY_N}$. Every $\Ker_N$ can be described by system of $|\BY_{N-1}|$ equations with $|\BY_N|$ indeterminates
\begin{equation}
\sum_{\mu = \lambda + \Box} a_{\mu} (w + c(\Box)) = 0, ~ \forall \lambda \in \BY_{N-1}
\end{equation}
where $\sum a_{\mu} \qet{\mu} \in \vspan{\BY_N}$ is an arbirtrary vector.

Our goal is to show that this system has rank equal to $|\BY_{N-1}|$ for every $w \in \BC$. One can define the order on $\BY_2$ by setting $(1,1) < (2)$, and then introduce the order on $\BY_{N+1}$ inductively from the order on $|\BY_N|$: the smallest are the $\lambda^{(1)}_{\Box}$ obtained from $\lambda \in \BY_N$ by adding a box in the first column and ordered as the elements $\lambda \in\BY_N$, then $\lambda^{(2)}_{\Box}$ obtained from $\lambda \in \BY_N$ by adding the box to the second column that haven't been counted yet, ordered analogically, such that $\lambda^{(2)}_{\Box} > \mu^{(2)}_{\Box}$ if $\mu^{(2)}_{\Box} \neq \eta^{(1)}_{\Box}$ for some $\eta \in \BY_N$ and $\lambda > \mu$ in $\BY_N$, then $\lambda^{(3)}_{\Box}$ and so on. Having the basis in every $\vspan{\BY_N}$, we have
\begin{equation}
D|_{\BY_N} =
\begin{pmatrix}
w - N     & w + 1       & 0         & \cdots & 0      \\
0         & w - N + 1   & w         & \cdots & 0      \\
0         & 0           & w - N + 2 & \cdots & 0      \\
\vdots    & \vdots      & \cdots    & \ddots & \vdots \\
0         & \cdots      & \cdots    & \cdots & w - 1  \\
\vdots    & \vdots      & \vdots    & \ddots & \vdots 
\end{pmatrix}
\end{equation}
The only thing to care is the diagonal $(D|_{\BY_N})_{ii}$ for $i \in [1, \BY_N]$ and the elements above it which are coefficients of diagrams $\mu \in \BY_{N+1}$ with the first, column larger or the same as the first column of $\lambda^{(1)}_{\Box}$. By adding one box we can't enlarge the first column for more than one box, so if $\mu_i \in \BY_{N+1}$ has the first column at least two boxes larger than the first column of $(\lambda_j)^{(1)}_{\Box},~ \lambda_j \in \BY_N$ then $(D|_{\BY_N})_{ij} = 0$. And if lenghts of their first columns are equal, so lenghts of other column differ, then $\mu_i$ is under the diagonal $(D|_{\BY_N})_{ii}$, because $\mu_i = (\eta)^{(1)}_{\Box}$ for $\eta \in \BY_N$ and $\eta$ stands after $\lambda_j$ because its form column is shorter.

If $w \notin \{1,2,..., N+1\}$ then all the diagonal $(D|_{\BY_N})_{ii}$ is fully nontrivial and $\mathrm{rk}(D|_{\BY_N}) = |\BY_{N-1}|$. Otherwise we have $w + 1 > 0$ and we can reorder all the diagrams in transponed order which is given on $\BY_2$ like this: $(1,1) > (2)$, and is defined inductively from $\BY_N$ on $\BY_{N+1}$ the way described before with only change of columns to rows so boxes are added to the $k$-th row. This helps us to get fully notrivial diagonal $(D|_{\BY_N})_{ii}$, and the rank is $|\BY_{N-1}|$.
\subsection{Kernel of $U$}
If $z = 0$ than $U$ acts as zero on $\BY_1$ so with $w = 0$, we obtain $w \vacuum = 0$ so that $\vacuum$ spans an one-dimensional representation.

If $z,w$ are not equal to zero we can prove that on $\BY_N,~ N \geqslant 1$ $U$ has the trivial kernel for all $z$. The proof of this is an induction on the number of hooks forming the diagram. If $e(\sum_{\lambda \in \BY_N} a_{\lambda} \qetl) = 0$, let's consider the coefficient of every $\mu \in \BY_{N+1}$ and prove that they are all zeros. This is because every Young diagram can be decomposed into a disjoint union of hooks of form $(M, 1, ... 1)$, see \cite{fulton}.

Let's start from diagrams consisting of only one hook $(\underbrace{1,1,...1}_{N})$. If $w = -N$ we immediatly proceed to the diagram $(2, \underbrace{1,1,...1}_{N - 2})$, otherwise we see $a_{(\underbrace{1,1,...1}_{N})} = 0$, because $\qet{(\underbrace{1,1,...1}_{N+1})}$ is counted with coefficient $a_{(\underbrace{1,1,...1}_{N})}(w - N) = 0$; the we proceed to $(2,\underbrace{1,1,...1}_{N-2})$. Then, if $w = -N + 1$ we omit this one and move on to $(3,\underbrace{1,1,...1}_{N-3})$, otherwise $\qet{2,(\underbrace{1,1,...1}_{N-1})}$ has the coefficient $a_{(2, \underbrace{1,1,...1}_{N-2} )}(w - N + 1) + a_{(\underbrace{1,1,...1}_{N})}(w + 1) = 0$, so because of $a_{(\underbrace{1,1,...1}_{N} )} = 0$ we find that $a_{(2,\underbrace{1,1,...1}_{N-2} )}$. Analogically we obtain that $a_{\lambda} = 0$ for all diagrams $\alpha$ consisting of one hook. 

Now let us make the step of induction knowing that $a_{\lambda} = 0$ for all $\lambda$ decomposed into $k$ hooks. Let us notice that if $z$ is equal to a coordinate of one of the particles of the Young diagram so that the correspondent $a_{\lambda}$ has the coefficient 0 in all linear combinations $\sum_{\lambda + \Box = \mu} a_{\lambda} (w + c(\Box))$ for all $\mu$ we just omit the consideration of $a_{\lambda}$ until some other hook in this diagram where $a_{\lambda}$ will have non-zero coefficient, because on higher hook levels we can always add more than box with various containment, so for every complex $z$ we get the situation where a linear combination has this $a_{\lambda}$ counted notrivially as a coefficient of a $\mu$; there we have $a_{\lambda} = 0$. This can be done for every $a_{\lambda}$ because for every Young diagram we can add the box at least two different ways. 

So for $z \neq 0$ we have the trivial coefficient of every $(k+1+\mu_1, k+1+\mu_2, ... , k+1+\mu_k, k+1),~ \mu_{k+1} \neq 0$ because of induction step we have $0 = a_{(k+1+\mu_1, k+1+\mu_2, ... , k+1+\mu_k, k)} z$, then $a_{(k+1+\mu_1, k+1+\mu_2, ... , k+1+\mu_k, k)} = 0$, then by the way described before we get the induction step proved. Else if $z = 0$ then we start from $(k+1+l, k+1, ... , k+1, k+1)$ and by $0 = a_{(k+l, k+1, ... , k+1, k+1)} (k + l)$, then we proceed analogically. 

\begin{proof}[Proof of Theorem \ref{decompose}]
	We see that for every $z,w \in \BC,~ N \geqslant 2,~ \dim \ker(D|_{\BY_N}) = |\BY_{N-1}|$, and $|\BY_{N-1}|$ basis vectors span Verma modules with the weight $zw + 2N$. The only problem is with $\vacuum$ and $\qet{\Box}$. If $z,w \neq 0$ then $\vacuum$ spans Verma module and $\qet{\Box} = U\vacuum$. If $z = w = 0$ then $\qet{Box}$ spans Verma module and $\vacuum$ spans the trivial one-dimensional representation. If $z = 0,~ w \neq 0$ then $U \vacuum = 0,~ D\qet{\Box} = \vacuum$ and $D \vacuum = 0$. So we have $\frac{U(\SL_2)_{\Kerov}\vacuum \oplus U(\SL_2)_{\Kerov}\qet{\Box}}{U \vacuum = 0,~ D\qet{\Box} = \vacuum, D \vacuum = 0}$. In the last case where $z \neq 0,~ w = 0$ we have relations $U \vacuum = \qet{\Box},~ D\qet{\Box} = 0, D \vacuum = 0$
\end{proof}

	\section{Kerov operators and Virasoro operators}
	\begin{theorem}
		\begin{enumerate}[label = {(\alph*)}]
			\item Kerov repersentation $\Kerov(z,w)$ is equivalent to subrepresentation of modified Virasoro algebra  \label{KerovA} $\{\LV_{-1},\LV_0,\LV_1\}(\frac{z+w}{2}, \frac{z-w}{2i})$. \\
			\item Rim-hook Kerov representation $\mathrm{RH}\Kerov(z,w)$ can be realised by the operators \label{KerovB} $\{\LV_{-r}, \LV_0, \LV_r\}(\frac{(z+w)r}{2}, \frac{z - w}{2i})$.
		\end{enumerate}
	\end{theorem}
	\begin{proof}[Proof of Theorem 3.1\ref{KerovA}]
	By acting with $L_{-1}$ we may get the formal sum of one step forward shifts and shifts of two different particles where one is moved $x$ leftwards and the other is moved $x+1$ rightwards. This sum has the monomial $a_{-1}a_0$ so that
	\begin{equation}
	a_{-1}a_0 \qetl = \alpha \sum_{\mu = \lambda + \Box} \qet{\mu}
	\end{equation}
	So let's try to undertstand how summands with one left shift $a_k,~ k > 0$ do behave. \label{intersect} One particle moves $X \rightarrow X+x+1$ and the second moves $Y \rightarrow Y-x$. If those two intervals intersect and $X + x + 1$ doesn't coincide with other interval ends like shown on this figure
	\begin{center}
		\begin{tikzpicture}[fill=black]
		%\draw[help lines] (-1,-2) grid (6,3);
		\path (-1, 0) node(x) {...}
		(0,0) node(a) [circle,draw,fill] {}
		(1,0) node(b) [circle,draw,fill] {}
		(2,0) node(c) [circle,draw] {}
		(3,0) node(d) [circle,draw,fill] {}
		(4,0) node(e) [circle,draw] {}
		(5,0) node(f) [circle,draw,fill] {}
		(6,0) node(g) [circle,draw,fill] {}
		(7,0) node(h) [circle,draw] {}
		(8,0) node(i) [circle,draw] {}
		(9,0) node(j) [circle,draw,fill] {}
		(10,0) node(k) [circle,draw] {}
		(11,0) node(y) {...};
		\draw (2,-0.5) node(lab1) {$Y - x$};
		\draw (3,-0.5) node(lab2) {$X$};
		\draw (5,-0.5) node(lab3) {$Y$};
		\draw (7,-0.5) node(lab4) {$X + x +1$};
		\draw[thick] (x) -- (a) -- (b) -- (c) -- (d) -- (e) -- (f) -- (g) -- (h) -- (i) -- (j) -- (k) -- (y);
		%\draw[thick,red,->] (a) |- +(1,3) -| (c) |- (b);
		\draw[thick,blue,->] (d) .. controls (5, 1) .. (h);
		\draw[thick,red ,->] (f) .. controls (3.5, -1) .. (c);
		\end{tikzpicture}
		
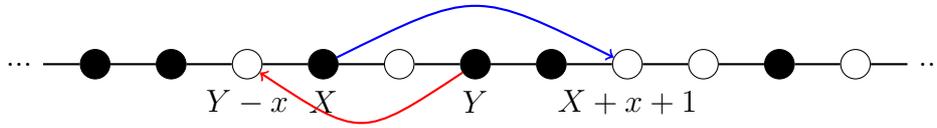
		\textsc{Figure 1: Intersecting intervals}
		\label{fig1}
	\end{center}
	then this pair of shifts is annihilated by the pair of shifts $X \rightarrow Y-x,~ Y \rightarrow X+x+1$
	\begin{center}
		\begin{tikzpicture}[fill=black]
		%\draw[help lines] (-1,-2) grid (6,3);
		\path (-1, 0) node(x) {...}
		(0,0) node(a) [circle,draw,fill] {}
		(1,0) node(b) [circle,draw,fill] {}
		(2,0) node(c) [circle,draw] {}
		(3,0) node(d) [circle,draw,fill] {}
		(4,0) node(e) [circle,draw] {}
		(5,0) node(f) [circle,draw,fill] {}
		(6,0) node(g) [circle,draw,fill] {}
		(7,0) node(h) [circle,draw] {}
		(8,0) node(i) [circle,draw] {}
		(9,0) node(j) [circle,draw,fill] {}
		(10,0) node(k) [circle,draw] {}
		(11,0) node(y) {...};
		\draw (2,-0.5) node(lab1) {$Y - x$};
		\draw (3,-0.5) node(lab2) {$X$};
		\draw (5,-0.5) node(lab3) {$Y$};
		\draw (7,-0.5) node(lab4) {$X + x +1$};
		\draw[thick] (x) -- (a) -- (b) -- (c) -- (d) -- (e) -- (f) -- (g) -- (h) -- (i) -- (j) -- (k) -- (y);
		%\draw[thick,red,->] (a) |- +(1,3) -| (c) |- (b);
		\draw[thick,blue,->] (d) .. controls (2.5, 0.5) .. (c);
		\draw[thick,red ,->] (f) .. controls (6, -0.5) .. (h);
		\end{tikzpicture}
		
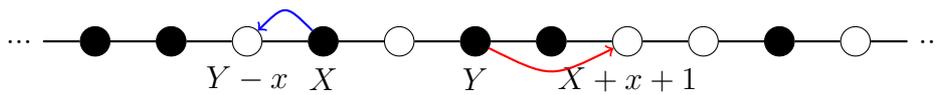
		\textsc{Figure 2: This is how intersecting jumps are resolved}
		\label{fig2}
	\end{center}
	because one monomial is counted with the sign $(-1)^{A+2B+C+1}$, and the second one has the coefficient $(-1)^{A+C}$. So one can shift a particle from $x$ one position right or "imitate" its shift by moving a particle placed in $y$ into $x+1$ and placing that particle to $y$. This imitation is illustrated below.
	\begin{center}
		\begin{tikzpicture}[fill=black]
		%\draw[help lines] (-1,-2) grid (6,3);
		\path (-1, 0) node(x) {...}
		(0,0) node(a) [circle,draw,fill] {}
		(1,0) node(b) [circle,draw,fill] {}
		(2,0) node(c) [circle,draw] {}
		(3,0) node(d) [circle,draw,fill] {}
		(4,0) node(e) [circle,draw] {}
		(5,0) node(f) [circle,draw,fill] {}
		(6,0) node(g) [circle,draw,fill] {}
		(7,0) node(h) [circle,draw] {}
		(8,0) node(i) [circle,draw] {}
		(9,0) node(j) [circle,draw,fill] {}
		(10,0) node(k) [circle,draw] {}
		(11,0) node(y) {...};
		\draw[thick] (x) -- (a) -- (b) -- (c) -- (d) -- (e) -- (f) -- (g) -- (h) -- (i) -- (j) -- (k) -- (y);
		%\draw[thick,red,->] (a) |- +(1,3) -| (c) |- (b);
		\draw[thick,blue,->] (j) .. controls (6, 1) .. node[above,sloped]{First jump} (e);
		\draw[thick,blue,->] (d) .. controls (6.5, -1) .. node[below,sloped]{Second jump} (j);
		\end{tikzpicture}
		
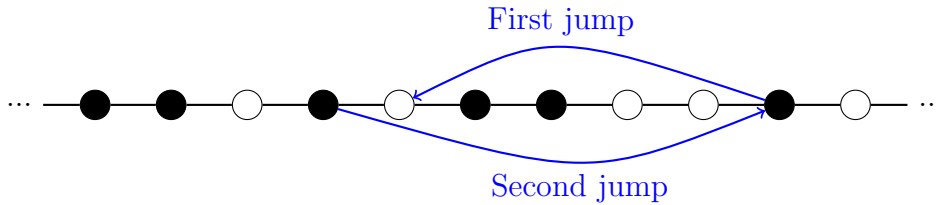
		\textsc{Figure 3: Imitation of moving particle 1 position rightwards}
		\label{fig3}
	\end{center}
	If we move the particle itself we can move it to every hole leftwards and then put it to the right place; these shifts are counted with coefficient $(-1)^{2A} = 1$. While imitating the shift we may take every particle right of our particle, these monomes have the sign $(-1)^{2A+1} = -1$. So particle shift $x \rightarrow x+1$ has the coefficient
	\begin{center}
		$\alpha + \#\{$ Holes left of x $\} - \#\{$ Particles right of x $\}$ 
	\end{center}
	The number of those particles is $\column(\Box) + 1$ and the number of those holes is $\row(\Box) + 1$ because of correspondence between Young diagrams and half-infinity particle configuratins as written in \cite{oko1}. Then the obtained coefficient is $z + c(\Box)$ by the definition.
	
	For $L_1$ we have the same calculations. Now we can consider the additional summand $i\beta ka_k$ and have
	\begin{equation}
	\begin{cases}
	\alpha + i\beta = z\\
	\alpha - i\beta = w
	\end{cases}
	\end{equation}
	There one can find $\alpha, \beta$ and conclude the proof.
	\end{proof}
	\begin{proof}[Proof of Theorem 3.1\ref{KerovB}]
	From Definition \ref{RHKerov} we may deduce
	\begin{equation}
	\begin{cases}
	U_r \qetl = \sum_{\mu = \lambda + \mathrm{Rim-hook}} (-1)^{\mathrm{height}+1}(z + \frac{1}{r^2}\sum_{\Box \in \mathrm{Rim-hook}} c(\Box) )\qet{\mu} \\
	D_r \qetl = \sum_{\mu = \lambda - \mathrm{Rim-hook}} (-1)^{\mathrm{height}+1}(z + \frac{1}{r^2}\sum_{\Box \in \mathrm{Rim-hook}} c(\Box) )\qet{\mu}
	\end{cases}
	\end{equation}
	Indeed, one shifts a particle $r$ positions rightwards and adds a $r$ box rim-hook to the Young diagram, because the shift changes one "down" to "up", levels up the next $r - 1$ intervals and changes the final "up" to "down". Rim-hook is connected, hence $c(\Box) \in [c(\LBox), c(\LBox) + r - 1]$, where $\LBox$ is the most left box added. Then $\frac{1}{r^2}\sum_{\Box \in \mathrm{Rim-hook}} c(\Box) = \frac{c(\LBox)}{r} + \frac12 - \frac{1}{2r} = \frac{c(\LBox) - \frac12}{r} + \frac12$, and $c(\LBox) - \frac12 = k$ because of containment definition. For $D_r$ check is analogous except we have to consider $c(\boxtimes) = k + \frac12$ where $\boxtimes$ is leftmost box of deleted rim-hook.
	
	Having defined Rim-hook Kerov operators action, we will consider $\LV_{-r}, \LV_0, \LV_r$ action on $\MY$. Notice that \hyperref[intersect]{the intersecting intervals argument} holds in this situation. For $r \geqslant 2$ we have such possibility 
	\begin{center}
		\begin{tikzpicture}[fill=black]
		%\draw[help lines] (-1,-2) grid (6,3);
		\path (-1, 0) node(x) {...}
		(0,0) node(a) [circle,draw,fill] {}
		(1,0) node(b) [circle,draw,fill] {}
		(2,0) node(c) [circle,draw] {}
		(3,0) node(d) [circle,draw,fill] {}
		(4,0) node(e) [circle,draw] {}
		(5,0) node(f) [circle,draw,fill] {}
		(6,0) node(g) [circle,draw,fill] {}
		(7,0) node(h) [circle,draw] {}
		(8,0) node(i) [circle,draw] {}
		(9,0) node(j) [circle,draw,fill] {}
		(10,0) node(k) [circle,draw] {}
		(11,0) node(y) {...};
		\draw (4,-0.5) node(lab1) {$Y - x$};
		\draw (1,-0.5) node(lab2) {$X$};
		\draw (6,-0.5) node(lab3) {$Y$};
		\draw (8,-0.5) node(lab4) {$X + x + r$};
		\draw[thick] (x) -- (a) -- (b) -- (c) -- (d) -- (e) -- (f) -- (g) -- (h) -- (i) -- (j) -- (k) -- (y);
		%\draw[thick,red,->] (a) |- +(1,3) -| (c) |- (b);
		\draw[thick,blue,->] (b) .. controls (4.5, 1) .. (i);
		\draw[thick,blue,->] (g) .. controls (5, -0.5) .. (e);
		\end{tikzpicture}
		
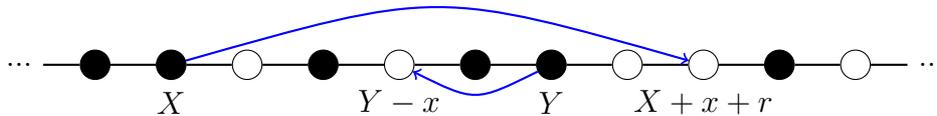
		\textsc{Figure 4: One jump interval included in another}
		\label{fig4}
	\end{center}
	These summands are counted with the sign $(-1)^{(A+B+1)+(B+1+C)} = (-1)^{A+C}$ and are annihilated by the summands of kind \begin{center}
		\begin{tikzpicture}[fill=black]
		%\draw[help lines] (-1,-2) grid (6,3);
		\path (-1, 0) node(x) {...}
		(0,0) node(a) [circle,draw,fill] {}
		(1,0) node(b) [circle,draw,fill] {}
		(2,0) node(c) [circle,draw] {}
		(3,0) node(d) [circle,draw,fill] {}
		(4,0) node(e) [circle,draw] {}
		(5,0) node(f) [circle,draw,fill] {}
		(6,0) node(g) [circle,draw,fill] {}
		(7,0) node(h) [circle,draw] {}
		(8,0) node(i) [circle,draw] {}
		(9,0) node(j) [circle,draw,fill] {}
		(10,0) node(k) [circle,draw] {}
		(11,0) node(y) {...};
		\draw (4,-0.5) node(lab1) {$Y - x$};
		\draw (1,-0.5) node(lab2) {$X$};
		\draw (6,-0.5) node(lab3) {$Y$};
		\draw (8,-0.5) node(lab4) {$X + x + r$};
		\draw[thick] (x) -- (a) -- (b) -- (c) -- (d) -- (e) -- (f) -- (g) -- (h) -- (i) -- (j) -- (k) -- (y);
		%\draw[thick,red,->] (a) |- +(1,3) -| (c) |- (b);
		\draw[thick,blue,->] (b) .. controls (2.5, 1) .. (e);
		\draw[thick,blue,->] (g) .. controls (7, -0.5) .. (i);
		\end{tikzpicture}
		
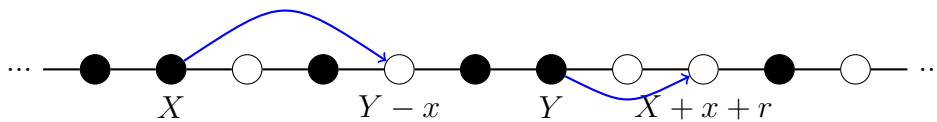
		\textsc{Figure 5: Resolving intersection}
		\label{fig5}
	\end{center}
	which have the sign $(-1)^{(A+C+1)+(B+C)} = (-1)^{A+C+1}$. Then we have the same Young diagrams we got from $\LV_{-1}$ action. It remains only to count the coefficients.
	
	We take out the general factor $(-1)^{height - 1}$ where $height$ is the number of particles in $(X,Y)$. Then the number of positive summands is
	\begin{center}
		$\#\{$ Holes left of $X$ $\} + \frac12 \#\{$ Particles in $(X, Y)$ $\}$ 
	\end{center}
	and the number of negative summands is
	\begin{center}
		$\#\{$ Particles right of $Y$ $\} + \frac12 \#\{$ Particles in $(X, Y)$ $\}$ 
	\end{center}    
	Then the coefficient is equal to
	$$
	\holes_{X} + \frac12 \holes_{(X,Y)} - \particles_Y - \frac12 \particles_{(X,Y)} = $$ 
	
	$$ = \holes_X - \particles_Y + \frac12 \holes_{(X,Y)} + \frac12 \particles_{(X,Y)} - \particles_{(X,Y)} = c(\LBox) + \frac{r - 1}{2}
	$$
	
	So we have
	\begin{eqnarray}
	\LV_{-r} \qetl = \sum_{\mu = \lambda + \mathrm{Rim-hook}} (-1)^{height + 1}(\alpha  - i\beta r + c(\LBox) + \frac{r - 1}{2}) \qet{\mu}\\
	\LV_{k}  \qetl = \sum_{\mu = \lambda - \mathrm{Rim-hook}} (-1)^{height + 1}(\alpha  - i\beta r + c(\LBox) + \frac{r - 1}{2})\qet{\mu}
	\end{eqnarray}
	Meanwhile
	\begin{eqnarray}
	U_r \qetl = \sum_{\mu = \lambda + \mathrm{Rim-hook}} (-1)^{height + 1}(z + \frac{c(\LBox)}{r} + \frac{r - 1}{2r}) \qet{\mu}\\
	D_r \qetl = \sum_{\mu = \lambda - \mathrm{Rim-hook}} (-1)^{height + 1}(w + \frac{c(\LBox)}{r} + \frac{r - 1}{2r})\qet{\mu}
	\end{eqnarray}
	Hence we have the condition of coincidence of those representations
	\begin{equation}
	\begin{cases}
	\alpha + ir\beta = rz\\
	\alpha - ir\beta = rw
	\end{cases}
	\end{equation}
	and $\alpha = \frac{(z+w)r}{2},~ \beta = \frac{z - w}{2i}$.
	\end{proof}
	\section{Virasoro process}
	We have proved that
	\begin{equation}
	\begin{cases}
	\LV_{-k} \qetl = \sum_{\mu = \lambda + \mathrm{Rim-hook}} (-1)^{\mathrm{height} - 1}(z + \Start + \frac{k}{2})\qet{\mu} \\
	\LV_{k} \qetl = \sum_{\mu = \lambda - \mathrm{Rim-hook}}(-1)^{\mathrm{height} - 1}(w + \Start + \frac{k}{2})\qet{\mu} 
	\end{cases}
	\end{equation}
	where $\Start$ is the inital coordinate of particle being moved by $\LV_{\pm k}$ and $\mathring{height}$ is a height of a rim-hook added to $\lambda$.
	\subsection{Definition and determinancy proof}
	After all this, we make this definition.
	\begin{definition}
		{\it Virasoro measure} or {\it Virasoro process} is a measure on Young diagrams defined by
		\begin{equation}
		\Vir(\lambda) = \frac{1}{Z} \triqet{\lambda}{\exp(\sum_{k \in \BN}x_k \LV_{-k})}{\emptyset}\triqet{\emptyset}{\exp(\sum_{k \in \BN}y_k \LV_{k})}{\lambda}
		\end{equation}
		where $\{x_k\}$ and $\{y_k\}$ are infinite sequences of complex numbers.
	\end{definition}
	\begin{prop} \label{psiL} On $\MY$ the following holds
		\begin{equation} \label{psiLcommute}
		[\psi_x, \LV_{-k}] = a_k\psi_x + (z + x + \frac{k}{2} - 1)\psi_{x+k}
		\end{equation}
	\end{prop}
	\begin{proof}
		From Theorem 3.1\ref{KerovB} we know that
		\begin{equation}
		\LV_{-k} w = \sum_{X \in \BHZ} (\holes_X - \particles_X + \frac{k-1}{2}) \psi_{X+k}\psi^*_X w
		\end{equation}
		where $w = \xi_1 \wedge \xi_2 \wedge ... \wedge \xi_r \wedge ...$ and sequence $\xi_1 > \xi_2 > ... > \xi_r > ...$ contains $(-\infty; N]$ for some $N$. This sequence can be represented as 
		$\lambda_1 - \frac12 + Q, \lambda_2 - \frac32 + Q, ... \lambda_R - \frac{R - 1}{2} + Q, -\frac{R+1}{2} - Q, ... -N-\frac{R+1}{2} - Q, ...$ for some Young diagram $\lambda$. When $x \neq X$, we can perform both $\psi_x\psi_{X+k}\psi^*_X$ and $\psi^*_X\psi_x\psi_{X+k}$ or none of them. So when we add the particle in $x$ in the first place, we decrease $\holes_X$ by 1 one increase $\particles_X$ by 1, and 
		\begin{equation}
		(\psi_x \LV_{-k} - \LV_{-k} \psi_x) w = \sum_{X \in \BHZ} \psi_{X+k}\psi^*_X \psi_x w
		\end{equation}
		This sum forms the first monomial in the right side of \ref{psiLcommute}.
		
		When $X = x$ and is actually a hole, then $\psi_{X+k}\psi^*_{X}$ acts trivially, but $\LV_k \psi_x$ actually puts the particle and then shifts it to $x + k$ without moving anything else. So we may assume that we place the particle in $x+k$. The coefficient is actually $z + x +\frac{k}{2} - 1$ because of Virasoro operators action and argument described in the beginning of the proof.  
	\end{proof}
	This proposition helps us to understand that exponents of linear combinations of Virasoro operators commute complicately. However, Young diagram space is rather small. The following theorem is to demonstrate this.
	\begin{theorem} \label{determinant}
		Virasoro process is determinantal; moreover, it can be described as a Schur process i. e. there exist sequences $\{X_i\},~ \{Y_i\}$ such that
		\begin{equation}
		\triqet{\lambda}{\exp(\sum_{k \in \BN}x_k \LV_{-k})}{\emptyset}\triqet{\emptyset}{\exp(\sum_{k \in \BN} y_k \LV_{k})}{\lambda} = \triqet{\lambda}{\exp(\sum_{k \in \BN}X_k a_{-k})}{\emptyset}\triqet{\emptyset}{\exp(\sum_{k \in \BN}Y_k a_{k})}{\lambda}
		\end{equation}
	\end{theorem}
	\begin{proof}
		Here we will just prove the very fact of determinancy. For every sequence $\{x_i\}$ there exists an another sequence $\{X_i\}$ such that $X_N = \triqet{\emptyset}{\exp(x_k \LV_{-k})}{(N)}$. This sequence can be calculated inductively beginning from $s_1 = y_1$. In this case it's $v_N = \Vir((N))$. These values define the Virasoro measure completely because the Jacobi-Trudy identity holds here
		\begin{equation}
		v_{\lambda} = \det[v_{\lambda_i - i + j}]
		\end{equation}
		Indeed, one can obtain $\lambda$ diagram from vacuum only by shifting first $|\lambda|$ particles, then the coefficient $v_{\lambda}$ is obtained from the sum of all possible shifts. But the action $\LV_{-k}$ on $\qetl$ gets the same sign on translation $\qetl \rightarrow \qet{\mu}$ as the action $a_{-k}$ an this sign is equal to $(-1)^{\particles - 1}$, where $\particles$ is the number of particles in jump interval. So Virasoro shifts production has the same sign as Heisenberg shift product does and the same determinant can be defined. 
		
		That was the $\triqet{\lambda}{\exp(x_k \LV_{-k})}{\emptyset}$; let's prove that the other factor can be rewritten using the exponent of Heisenberg operators linear combination. We see that $\LV^*_{k} \neq \LV_{-k}$, but we can treat $\exp(\sum_i y_i \LV^*_i)$ the same way we treated $\exp(x_k \LV_{-k})$ because we can define $\LV^*_i$ action on Young diagrams by the definition:
		\begin{equation}
		\triqet{\mu}{\LV_k}{\lambda} = \triqet{\lambda}{\LV^*_k}{\mu} 
		\end{equation}
		
		Here we conclude that Virasoro process is just a Schur process we know from \cite{oko2}. Okounkov has proved \cite{oko2} that Schur process is determinantal, hence Virasoro process is determinantal.
	\end{proof}
	\subsection{From Virasoro process to Schur process}
	\begin{definition}
		For the path $\Start \rightarrow \Start + k_1 \rightarrow \Start + k_1 + k_2 \rightarrow ... \rightarrow \Start + k_1 + ... + k_R$ on $\BHZ$ we introduce the {\it path polynomal}
		\begin{equation}
		W_z(k_1, ... k_R) = (z + \Start + k_1)(z + \Start + k_1 + \frac{k_2}{2})...(z + \Start + k_1 + ... k_{R-1} + \frac{k_R}{2})
		\end{equation}
	\end{definition}
	Because we work with the standart Young diagrams we set $Start = -\frac12$ as default. If we have $\exp(\sum_k y_k \LV_{-k})$, we know that
	\begin{equation}
	\exp(\sum_{k \in \BN} x_k \LV_{-k}) \vacuum = \sum_{\lambda \in \BY} \Vir_{\lambda}(z, x_i) \qetl
	\end{equation} 
	and
	\begin{equation}
	\Vir_{(N)}(z, x_i) = \Vir_N(z, x_i) = \sum_{\sum i k_i = N} \frac{1}{(\sum k_i)!} \prod x_i^{k_i} \sum_{\sigma \in S_N / \mathrm{Stab}} \sigma(W_z(k_1, ... k_R))
	\end{equation}
	We will try to calculate $X_i$ such that $\exp(\sum_{i \in \BN} X_i a_{-i}) \vacuum = \exp(\sum{i \in \BN} x_i \LV_{-i})\vacuum$.
	\begin{theorem} \label{zlinearity}
		$X_N = A_N z + B_N$, where $A_N,~ B_N$ are some polynomials in indeterminates $x_i$.
	\end{theorem}
	\begin{proof}
		Let's prove it by induction. Base step is got immediatly: $X_1 = x_1 z,~ X_2 = (\frac{x_1^2}{2} + x_2)z + \frac{x_2}{2}$. Then we prove the induction step from $N$ to $N+1$. On the left side of identity,
		\begin{equation}
		\frac{d}{dz} s_{N+1}(X_1, ... ) = X_{N+1}' + \sum_{k = 0}^{N} s_{N-k}(X_1, ...) X_{k+1}'
		\end{equation}
		because of $s_{N+1} = \sum_{\sum i k_i = N+1} \frac{X_1^{k_1}}{k_1!}\frac{X_2^{k_2}}{k_2!}...\frac{X_{N+1}^{k_{N+1}}}{k_{N+1}!}$, differentiation of each monomial gives us
		\begin{equation}
		\frac{X_1^{k_1}}{k_1!}\frac{X_2^{k_2}}{k_2!}...\frac{X_r^{k_r-1}}{(k_r-1)!}...\frac{X_{N+1}^{k_{N+1}}}{k_{N+1}!}X_r'
		\end{equation}
		and the coefficient of $X_r'$ is equal to $\sum_{\sum i k_i = N+1 - r} \frac{X_1^{k_1}}{k_1!}\frac{X_2^{k_2}}{k_2!}...\frac{X_{N+1}^{k_{N+1}}}{k_{N+1}!}$. Knowing that and the induction hypothesis we change $s_k$ to $\Vir_k$ and compare this with the right side derivate.
		
		On the right side
		\begin{equation}
		\frac{d}{dz}\Vir_{N+1}(z, x_i) = \sum_{\sum i k_i = N} \frac{1}{(\sum k_i)!} \prod x_i^{k_i} \sum_{\sigma \in S_N / \mathrm{Stab}} \sigma\frac{d}{dz}(W_z(k_1, ... k_R))
		\end{equation}
		We reduce $\frac{d}{dz}(W_z(k_1, ... k_R)$ to a linear combination of various $W_z$. Because of Leibnitz rule
		\begin{equation}
		\frac{d}{dz}W_z(k_1, ... k_R) = \sum_{r} \frac{W_z(k_1, ... k_R)}{(z + \Start + k_1 + ... k_{r-1} + \frac{k_r}{2})}
		\end{equation}
		We'll treat every summand the way described below
		$$
		\frac{W_z(k_1, ... k_R)}{(z + \Start + k_1 + ... k_{r-1} + \frac{k_r}{2})} =
		$$
		
		$$
		= W_z(k_1, ... k_{r-1})(z + \Start + k_1 + ... + k_r + \frac{k_{r+1}}{2})...(z + \Start + k_1 + ... + k_{R-1} + \frac{k_R}{2}) = 
		$$
		
		$$
		= W_z(k_1, ... k_{r-1})((z + \Start + k_1 + ... + \frac{k_{r+1}}{2}) + k_r)...((z + \Start + k_1 + ... + k_{R-1} + \frac{k_R}{2}) + k_r)
		$$
		We take out $W_z(k_1, ... k_{r-1})(z + \Start + k_1 + ... + \frac{k_{r+1}}{2})...(z + \Start + k_1 + ... + k_{R-1} + \frac{k_R}{2})$, which is $W_z(k_1, ... k_{r-1}, k_{r+1}, ... k_R)$. Other summands contain less factors of form $(z + \Start + k_1 + ... k_{r-1} + k_{r+1} + ... + k_l + \frac{k_{l+1}}{2})$, those polynomials are reduced like this
		$$
		\frac{k_r W_z(k_1, ... k_{r-1})(z + \Start + k_1 + ... + \frac{k_{r+1}}{2})...(z + \Start + k_1 + ... + k_{R-1} + \frac{k_R}{2})}{z + \Start + k_1 + ... + k_{l-1} + \frac{k_{l}}{2}} = 
		$$
		
		$$
		= k_r W_z(k_1, ... k_{r-1}, k_{r+1}, ... k_{l-1})((z + \Start + k_1 + ... + k_{l - 1} + \frac{k_{l+1}}{2}) + k_r + k_l)...
		$$
		
		$$
		...((z + \Start + k_1 + ... + k_{R-1} + \frac{k_R}{2}) + k_r + k_l)
		$$
		Here we can pick out $k_r W_z(k_1, ... k_{r-1}, k_{r+1}, ... k_{l-1}, k_l, ... k_R)$ and continue this procedure until the linear combination of way polynomeials is formed. Hence we have
		\begin{equation}
		\frac{d}{dz}W_z(k_1, ... k_R) = \sum_{1 \leqslant r_1 < r_2 < ... < r_l \leqslant R} k_{r_2} (k_{r_2} + k_{r_3}) ... (k_{r_2} + ... + k_{r_l}) W_z (k_1, ... \widehat{k_{r_1}}, ... \widehat{k_{r_l}}, ... k_R)
		\end{equation}
		\begin{definition}
			For the set $\{\xi_1, ... \xi_M\}$ we call the subgroup of permutations $\sigma \in S_M$ such that $\xi_i = \xi_{\sigma(i)},~ \forall i$ the {\it stabilizer of $\{\xi_1, ... xi_M\}$} and denote it by $Stab(\xi_1, ... \xi_M)$.
		\end{definition}
		Now if the set $\{k_1, ... \widehat{k_{r_1}}, ... \widehat{k_{r_l}}, ... k_R\}$ is fixed then $W_z (k_1, ... \widehat{k_{r_1}}, ... \widehat{k_{r_l}}, ... k_R)$ is included in the sum with coefficient
		\begin{equation}
		\binom{R}{l}\binom{l}{\#_1, ... \#_l} \sum_{\sigma \in S_l/\mathrm{Stab}(k_{r_1}, ... k_{r_l})} \sigma(k_{r_2} (k_{r_2} + k_{r_3}) ... (k_{r_2} + ... + k_{r_l}))
		\end{equation}
		for the fixed $k_{r_1}, ... k_{r_l}$ where $\#_i$ is the number of $k_{r_j}$ takes the $i$-th value (all these values are ordered by maximality) (if there are less than $R$ different values then beginning from some moment $\#_i = 0$). Hence this coefficient doesn't depend on the set $k_{r_1}, ... k_{r_l}$ and the correspondent polynome
		\begin{equation}
		\frac{x_1^{g_1}...x_{N+1}^{g_{N+1}}}{(R-l)!} W_z (k_1, ... \widehat{k_{r_1}}, ... \widehat{k_{r_l}}, ... k_R)
		\end{equation}
		is a summand of the Schur polynome $s_{N+1 - \sum_i k_{r_i}}$ according to induction hypothesis. Symmetric group action on ways permuting the jumps of different lenght allows to obtain all the way polynomes. And the scalar factor of $s_{N+1 - \sum_i k_{r_i}}$ in $X'_{N+1}$ is equal to factor of $s_{N - \sum_i k_{r_i}}$ in $X'_{N}$ because these two coefficients are obtained by the same way. Hence because $\{x'_{i}\}|_{i\in [1;N]}$ don't depend on $z$ and they are polynomes of $z$ we get $x'_{N+1}$ doesn't depend on $z$ and $X_N = A_N z + B_N,~ N \in \BN$.
	\end{proof}
	Of course $A_N$ and $B_N$ can be calculated algorithmically. The proof of the last theorem allows to get the formula for $A_N$ as a $s_1$ coefficient in $s'_{N+1}$. It can be obtained by cutting all the jumps except the last one from all the ways ending with jump of length 1. This can be done by taking derivative by the first jump and excluding all others as described in the proof. Here we get
	\begin{equation}
	A_N = \sum_{\sum_i k_i = N} \left[\frac{y_{\#_1}^{k_{\#_1}}...y_{\#_R}^{k_{\#_R}}}{R!}\binom{R}{\#_1, ... \#_R}\sum_{\sigma \in S_R / \mathrm{Stab}(k_{r_1}, ... k_{r_l})} \sigma(k_2(k_2 + k_3)...(k_2 + ... + k_R))\right]
	\end{equation}
	where $\#_1 < \#_2 < ... < \#_R$ is the set of values of $\{k_i\}$.
	Also we have $s_N(B_1, ... B_n ...) = v_N(z = 0, x_1, ... x_n ...)$, hence $1 + \sum_{N = 1}^{\infty} v_N u^N = \exp(1 + \sum_{n = 1}^{\infty}B_n u^n)$ so by applying the series expansion for logarithm we have
	\begin{equation}
	B_N = \sum_{\sum_i l_i = N} (-1)^{n-1} \frac{\binom{N}{l_1, ... , l_n} v_{l_1} ... v_{l_n}}{n}
	\end{equation}
	where $v_{l} = \Vir(\lambda)|_{z = 0}$.
	
	For $\exp(\sum_{i \in \BN} y_i \LV_i)$ we have the same theorem hold:
	\begin{prop}
		$Y_N = C_N w + D_N$, where $C_N,~D_N$ are some polynomials in indeterminates $y_i$.
	\end{prop}
	\begin{proof}
		We have
		\begin{equation}
		\triqet{\emptyset}{\exp(y_k \LV_{k})}{\lambda} = \sum_{\sum i k_i = N} \frac{1}{(\sum k_i)!} \prod y_i^{k_i} \sum_{\sigma \in S_N / \mathrm{Stab}} \sigma(W_z(k_1, ... k_R))
		\end{equation}
		so the proof is the same as the proof of Theorem \ref{zlinearity}.
	\end{proof}
	\section{M-Virasoro process}
	Finally, we would like to say some words about some generalizations of Kerov operators construction. 
	\begin{definition}
		{\it M-Virasoro operators} are operators defined by
		\begin{equation}
		\LV^{(M)}_k = i\beta k a_k + \frac{1}{M!}\sum_{\sum k_i = k} |Stab(k_1, ... k_M)| : a_{k_1} a_{k_2} ... a_{k_M} :
		\end{equation}
		where $Stab(k_1, ... k_M) \subset S_M$ is the stabilizer of $\{k_i\}$ defined above.
	\end{definition}
	
	We use this stabilizer to count every trajectory once.
	
	\begin{center}
		\begin{tikzpicture}[fill=black]
		%\draw[help lines] (-1,-2) grid (6,3);
		\path (-1, 0) node(x) {...}
		(0,0) node(a) [circle,draw,fill] {}
		(1,0) node(b) [circle,draw,fill] {}
		(2,0) node(c) [circle,draw] {}
		(3,0) node(d) [circle,draw,fill] {}
		(4,0) node(e) [circle,draw] {}
		(5,0) node(f) [circle,draw,fill] {}
		(6,0) node(g) [circle,draw,fill] {}
		(7,0) node(h) [circle,draw] {}
		(8,0) node(i) [circle,draw] {}
		(9,0) node(j) [circle,draw,fill] {}
		(10,0) node(k) [circle,draw] {}
		(11,0) node(y) {...};
		\draw[thick] (x) -- (a) -- (b) -- (c) -- (d) -- (e) -- (f) -- (g) -- (h) -- (i) -- (j) -- (k) -- (y);
		%\draw[thick,red,->] (a) |- +(1,3) -| (c) |- (b);
		\draw[thick,blue, ->] (b) .. controls (2.5, 1) .. (e);
		\draw[thick,blue, ->] (e) .. controls (4.5, 0.5) .. (f);
		\draw[thick,blue, ->] (f) .. controls (2.5, -0.5) .. (a);
		\draw[thick,blue ,->] (a) .. controls (4.5, -1.25) .. (i);
		\end{tikzpicture}
		
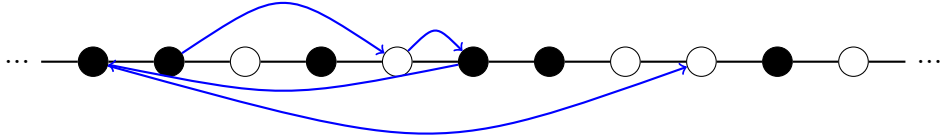
		\textsc{Figure 6: Example of trajectory}
		\label{fig6}
	\end{center}
	
	In some way M-Virasoro operators are similar to basic Virasoro operators: one can see that M-Virasoro operator $\LV^{(M)}_{-k}$ can shift only one particle by $k$ positions rightwards and $\LV^{(M)}_{k}$ can shift only one particle by $k$ positions leftwards. Otherwise we have at least two particle trajectories and we can swap these trajectories ends to get this very summand with the other sign just as we had before for Virasoro operators. The example of this swapping is shown below.
	\begin{center}
		\begin{tikzpicture}[fill=black]
		%\draw[help lines] (-1,-2) grid (6,3);
		\path (-1, 0) node(x) {...}
		(0,0) node(a) [circle,draw,fill] {}
		(1,0) node(b) [circle,draw,fill] {}
		(2,0) node(c) [circle,draw] {}
		(3,0) node(d) [circle,draw,fill] {}
		(4,0) node(e) [circle,draw] {}
		(5,0) node(f) [circle,draw,fill] {}
		(6,0) node(g) [circle,draw,fill] {}
		(7,0) node(h) [circle,draw] {}
		(8,0) node(i) [circle,draw] {}
		(9,0) node(j) [circle,draw,fill] {}
		(10,0) node(k) [circle,draw] {}
		(11,0) node(y) {...};
		\draw[thick] (x) -- (a) -- (b) -- (c) -- (d) -- (e) -- (f) -- (g) -- (h) -- (i) -- (j) -- (k) -- (y);
		%\draw[thick,red,->] (a) |- +(1,3) -| (c) |- (b);
		\draw[thick,blue, ->] (b) .. controls (2.5, 1) .. (e);
		\draw[thick,blue, ->] (e) .. controls (4.5, 0.5) .. (f);
		\draw[thick,blue, ->] (f) .. controls (2.5, 0.5) .. (a);
		\draw[thick,blue ,->] (a) .. controls (4.5, 1.25) .. (i);
		\draw[thick,blue ,->] (d) .. controls (4.5, -1.25) .. (g);
		\draw[thick,blue ,->] (g) .. controls (8.5, -1.25) .. (k);
		\draw[thick,blue ,->] (k) .. controls (9.5, -1.25) .. (j);
		\draw[thick,blue ,->] (j) .. controls (8, -1.25) .. (h);
		\end{tikzpicture}
		
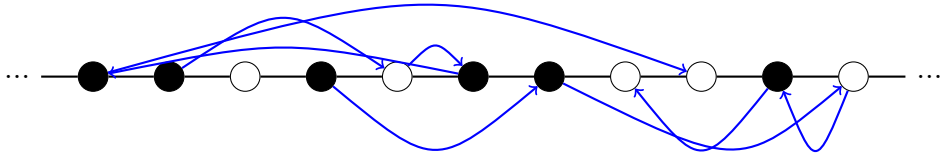
		\textsc{Figure 7: Two different trajectories}
		\label{fig7}
	\end{center}
	\begin{center}
		\begin{tikzpicture}[fill=black]
		%\draw[help lines] (-1,-2) grid (6,3);
		\path (-1, 0) node(x) {...}
		(0,0) node(a) [circle,draw,fill] {}
		(1,0) node(b) [circle,draw,fill] {}
		(2,0) node(c) [circle,draw] {}
		(3,0) node(d) [circle,draw,fill] {}
		(4,0) node(e) [circle,draw] {}
		(5,0) node(f) [circle,draw,fill] {}
		(6,0) node(g) [circle,draw,fill] {}
		(7,0) node(h) [circle,draw] {}
		(8,0) node(i) [circle,draw] {}
		(9,0) node(j) [circle,draw,fill] {}
		(10,0) node(k) [circle,draw] {}
		(11,0) node(y) {...};
		\draw[thick] (x) -- (a) -- (b) -- (c) -- (d) -- (e) -- (f) -- (g) -- (h) -- (i) -- (j) -- (k) -- (y);
		%\draw[thick,red,->] (a) |- +(1,3) -| (c) |- (b);
		\draw[thick,blue, ->] (b) .. controls (2.5, 1) .. (e);
		\draw[thick,blue, ->] (e) .. controls (4.5, 0.5) .. (f);
		\draw[thick,blue, ->] (f) .. controls (2.5, 0.5) .. (a);
		\draw[thick,red ,->] (a) .. controls (4.5, 1.25) .. (h);
		\draw[thick,blue ,->] (d) .. controls (4.5, -1.25) .. (g);
		\draw[thick,blue ,->] (g) .. controls (8.5, -1.25) .. (k);
		\draw[thick,blue ,->] (k) .. controls (9.5, -1.25) .. (j);
		\draw[thick,red ,->] (j) .. controls (8.5, -1) .. (i);
		\end{tikzpicture}
		
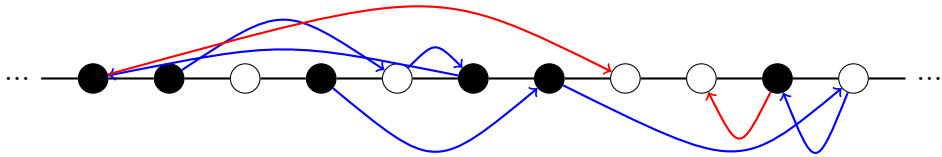
		\textsc{Figure 8: Swapping the endings of trajectories}
		\label{fig8}
	\end{center}
	Then one can obatain
	\begin{equation}
	L^{(M)}_k \qetl = \sum_{\mu = \lambda + \mathrm{Rim-hook}} (-1)^{height - 1} (z + c(\LBox) + \frac{k-1}{2})^{M-1} \qet{\mu}
	\end{equation}
	because every trajectory is counted once, every trajectory contains some particles shifted alongside it and trajectories with $r$ particles are counted with coefficient $(-1)^r \binom{M-1}{r,r_0,M-1-r-r_0}z^{r_0}(\particles)^r(\holes)^{M-1-r-r_0}$, where $\particles$ is number of particles after $X$ and $\holes$ is the number of holes before $X+k$. 
	
	An analogue of proposition \ref{psiL} holds for M-Virasoro operators:
	\begin{prop}
		On $\MY$ the following holds
		\begin{equation}
		[\psi_x, \LV^{(M)}_{-k}] = (\sum_{r = 1}^{M-1} \LV^{(M-r)})\psi_x + (z + x + \frac{k}{2} - 1)^{M-1}\psi_{x+k}
		\end{equation}
	\end{prop}
	\begin{proof}
		We prove it the same way as we proved Proposition \ref{psiL} but we put there
		\begin{equation}
		(z + x + \frac{k}{2})^{M-1} - (z + x + \frac{k}{2} - 1)^{M-1} = \sum_{r = 1}^{M-1} (z + x + \frac{k}{2} - 1)^{M-r}
		\end{equation}
		and therefore we can consider that coefficient $(z + x + \frac{k}{2} - 1)^{M-r}$ appears from $\LV^{(M-r)}_k$ action.
	\end{proof}
	
	\begin{definition}
		{\it M-Virasoro measure} is the measure on Young diagrams defined as
		\begin{equation}
		\Vir^{(M)}(\lambda) = \frac{1}{Z} \triqet{\lambda}{\exp(\sum_{k \in \BN}x_k \LV^{(M)}_{-k})}{\emptyset}\triqet{\emptyset}{\exp(\sum_{k \in \BN}y_k \LV^{(M)}_{k})}{\lambda}
		\end{equation} 
	\end{definition}
	
	When $M = 1$ then M-Virasoro process is just the Schur process, and if $M = 2$ then M-Virasoro process coincides with Virasoro process. For M-Virasoro process Theorem \ref{determinant} is correct because of the same reasons. However, the calculations of correlation functions are tough and we'll complete them elsewhere.
	\section{z-meausures for classical Lie algebras}
	We have considered the measures of kind
	\begin{equation}
	\FM(g_1,g_2, v) = \frac{1}{Z}\triqet{v_0}{g_1}{v}\triqet{v}{g_2}{v_0}
	\end{equation}
	where $g_1 \in \exp(\g_1),~ g_2 \in \exp(\g_2)$ are the elememts of two different Lie groups, $v, v_0$ are from some representation space common for $\g_1$ and $\g_2$, where $v_0$ is the cyclic vector. Schur measures are constructed for $\g_1 = \mathrm{Heis}_+ = \vspan{\{a_i\}_{i \in \BN}},~ \g_2 = \mathrm{Heis}_- = \vspan{\{a_{-i}\}_{i \in \BN}}$ and for Virasoro measures we have $\g_1 = \Vir_+ = \vspan{\{L_i\}_{i \in \BN}},~ \g_2 = \Vir_- = \vspan{\{L_{-i}\}_{i \in \BN}}$. The representation space there is $\Lambda^{\frac{\infty}{2}}V$ and $v_0 = \vacuum$ is the cyclic vector. Note that $\exp(\Vir_+)$ and $\exp(\Vir_-)$ aren't embeddable in a single Lie group.
	
	Let us now consider $\g = \n_- \oplus \h \oplus n_+$ as the classical Lie algebra with $\n_-$ and $\n_+$ are Borel subalgebras and $\h$ as a Cartan algebras. We'll consider z-measure for $\g$ this way:
	\begin{equation}
	\FM(n_+,n_-, v) = \frac{1}{Z}\triqet{v_0}{n_+}{v}\triqet{v}{n_-}{v_0}, ~~n_{\pm} \in \n_{\pm}
	\end{equation}
	Firstly, we consider the case $\g = \sl_{n+1}$ of $A_n$ type.
	\section{Acknowledgements}
	I would like to thank Evgeny Feigin for useful remarks and discussions. I also would like to thank Alexey Barsukov and Mikhail Artemyev. 
 
%\bigskip\addtocontents{toc}{\medskip}

\hbadness=1100

\end{document}